\documentclass[12pt,a4paper]{amsart}
\usepackage[top=40mm, bottom=35mm, left=35mm, right=35mm]{geometry}
\usepackage{mathptmx}
\usepackage{mathrsfs}

\usepackage{verbatim}
\usepackage{url}
\usepackage[all]{xy}
\usepackage{color}
\usepackage{hyperref}

\theoremstyle{plain}
\newtheorem{thm}{Theorem}[section]
\newtheorem{lem}[thm]{Lemma}
\newtheorem{prop}[thm]{Proposition}

\theoremstyle{definition}
\newtheorem{de}[thm]{Definition}
\newtheorem{rem}[thm]{Remark}
\newtheorem{ex}[thm]{Example}
\newtheorem{ques}[thm]{Question}
\newcommand{\Z}{{\mathbb{Z}_+}}
\newcommand {\R} {\mathbb R}
\newcommand{\N}{\mathbb{N}}

\newcommand {\ol} {\overline}

\newcommand{\ep}{\varepsilon}

\newcommand{\lra}{\longrightarrow}

\begin{document}
\title[Recurrence properties and disjointness on the induced spaces] {Recurrence properties and disjointness on the induced spaces}
\author[J. Li]{Jie Li}
\thanks{The first and third authors are supported by NNSF of China (11371339). The second author
is partially supported by NNSF of China (11161029, 11171320) and NSF for Young Scholar of Guangxi Province (013GXNSFBA019020)}
\date{\today}
\address[J.~Li]{Jie Li, Wu Wen-Tsun Key Laboratory of Mathematics, USTC, Chinese Academy of Sciences and
School of Mathematics, University of Science and Technology of China,
Hefei, Anhui, 230026, P.R. China}
\email{jiel0516@mail.ustc.edu.cn}
\author[K. Yan]{Kesong Yan}
\address[K. Yan]{Wu Wen-Tsun Key Laboratory of Mathematics, USTC, Chinese Academy of Sciences and
School of Mathematics, University of Science and Technology of China,
Hefei, Anhui, 230026, P.R. China}
\address{Department of Mathematics and Computer Science, Liuzhou Teachers College,
Liuzhou 545004, Guangxi, P.R. China}
\email{ksyan@mail.ustc.edu.cn}
\author[X.~Ye]{Xiangdong Ye}
\address[X.~Ye]{Wu Wen-Tsun Key Laboratory of Mathematics, USTC, Chinese Academy of Sciences and
School of Mathematics, University of Science and Technology of China,
Hefei, Anhui, 230026, P.R. China}
\email{yexd@ustc.edu.cn}
\begin{abstract}

A topological dynamical system induces two natural systems, one is
on the hyperspace and the other one is on the space of probability measures. The
connection among some dynamical properties on the original space and
on the induced spaces are investigated. Particularly, a minimal
weakly mixing system which induces a $P$-system on the probability measures
space is constructed and some disjointness result is obtained.
\end{abstract}

\maketitle

\section{Introduction}
Throughout this paper, by a {\it topological dynamical system}
(t.d.s. for short) we mean a pair $(X, T)$, where $X$ is a compact
metric space with a metric $\rho$ and $T$ is a continuous surjective
map from $X$ to itself.  A non-empty closed invariant subset $Y
\subset X$ (i.e., $TY \subset Y$) defines naturally a subsystem $(Y,
T)$ of $(X, T)$.

\medskip

A t.d.s. $(X,T)$ induces two natural systems, one is $(K(X),T_K)$ on
the hyperspace $K(X)$ consisting of all non-empty closed subsets of
$X$ with the Hausdorff metric, and the other one is $(M(X),T_M)$ on
the probability measures space $M(X)$ consisting of all Borel probability
measures with the weak*-topology. Bauer and Sigmund \cite{BS75}
first gave a systematic investigation on the connection of dynamical
properties among $(X,T)$, $(K(X),T_K)$ and $(M(X),T_M)$. It was
proved that $(X,T)$ is weakly mixing (resp. mildly mixing, strongly
mixing) if and only if $(K(X),T_K)$ (or $(M(X),T_M)$) has the
same property. We remark that later it was shown that the
transitivity of $(K(X),T_K)$ (resp. $(M(X),T_M)$) is equivalent to
the weak mixing property of $(K(X),T_K)$ (resp. $(M(X),T_M)$), see
\cite{Ban05} and \cite{Shao}.

\medskip

Since then the connection of dynamical properties among $(X,T)$,
$(K(X),T_K)$ and $(M(X),T_M)$  has been studied by many authors,
see, e.g., \cite{Ban05, GW95, GKLOP09, KL1,  KL2,  Kom84,Li13,
Rom03}. A remarkable result by Glasner and Weiss \cite{GW95} stated
that the topological entropy of $(X,T)$ is zero if and only if so is
$(M(X), T_M)$  (similar results related to nullness and tameness can
be found in \cite{KL1, KL2}). Recently, Li \cite{Li13} observed that
$(K(X), T_K)$ is a $P$-system if and only if $(X,T)$ is a weakly
mixing system with dense small periodic sets (first defined by Huang
and Ye in \cite{HY05} and called an HY-system in \cite{Li11}).

\medskip

In this paper we further exploit the connection, and focus on
periodic systems, $P$-systems,  $M$-systems, $E$-systems and
disjointness. On the way to do this we define an almost HY-system
and show that $(M(X),T_M)$ is a $P$-system if and only if $(X,T)$ is
a weakly mixing almost-HY-system (Theorem \ref{thm:4-10}). A minimal
weakly mixing system which induces a $P$-system on $(M(X),T_M)$ is
constructed showing that HY-systems and almost HY-systems are
different property (Theorem \ref{thm:4-11}). We conjecture that
there is a weakly mixing proximal $E$-system inducing a $P$-system
on $(M(X),T_M)$ (this will be answered affirmatively in a forthcoming 
paper \cite{LSY13}). See the following two tables for the further
connection (see Section 3 and Section 4 for details).

\begin{table}[!h]
\caption{The connection with hyperspace} \vspace*{1.5pt}
\begin{center}
\begin{tabular}{| p{2.0cm}| p{3.3cm} | p{2.0cm} | p{2.0cm} |p{2.0cm}|p{2.0cm}|}
\hline
\multicolumn{1}{|c|} {$(X,T)$}  & \multicolumn{1}{|c|}
{periodic} & \multicolumn{1}{|c|} {HY-system} & \multicolumn{1}{|c|}
{w.m. $M$-system} & \multicolumn{1}{|c|} {w.m. $E$-system}  \\
\hline
\multicolumn{1}{|c|} {$(K(X),T_K)$}  & \multicolumn{1}{|c|}
{pointwise periodic} & \multicolumn{1}{|c|} {$P$-system} & \multicolumn{1}{|c|}
{$M$-system}& \multicolumn{1}{|c|} {$E$-system}  \\
\hline
\end{tabular}
\end{center}
\end{table}

\begin{table}[!h]
\caption{The connection with probability measures space} \vspace*{1.5pt}
\begin{center}
\begin{tabular}{| p{2.0cm}| p{3.3cm} | p{2.0cm} | p{2.0cm} |p{2.0cm}|p{2.0cm}|}
\hline \multicolumn{1}{|c|} {$(X,T)$} & \multicolumn{1}{|c|}
{periodic} & \multicolumn{1}{|c|} {almost HY-sys.} &
\multicolumn{1}{|c|} {not nece. $M$-sys.} &
\multicolumn{1}{|c|} {w.m. $E$-sys.}  \\
\hline
\multicolumn{1}{|c|} {$(M(X),T_M)$} & \multicolumn{1}{|c|} {p.w. periodic} &
\multicolumn{1}{|c|} {$P$-system} & \multicolumn{1}{|c|} {$M$-system} &
\multicolumn{1}{|c|} {$E$-system}  \\
\hline
\end{tabular}
\end{center}
\end{table}

\medskip

The notion of disjointness of two t.d.s. was introduced by
Furstenberg in his seminar paper \cite{Fur67}.  It is known if two
t.d.s. are disjoint then one of them is minimal. It is an open
question which system is disjoint from all minimal system. It was
shown that if a transitive t.d.s. is disjoint from all minimal
system then it is weakly mixing with dense minimal points
\cite{HY05}, and a weakly mixing system with dense distal points is
disjoint from all minimal systems \cite{DSY12, O10}. In this paper
we show that if $(K(X),T_K)$ is disjoint from all minimal system,
then so is $(X,T)$ (Theorem \ref{thm:5-2}). It seems that there are
examples $(X,T)$ which do not have dense distal points and at the
same time $(K(X),T_K)$ do. Unfortunately we could not provide one at
this moment.

\vspace{4mm}

\noindent {\bf Acknowledgments.} We thank W. Huang, Jian Li and S.
Shao for very useful discussions. Particularly, we thank B. Weiss
for his valuable suggestion related to Theorem \ref{thm:4-11}.

\section{Preliminary}

\subsection{Basic definitions and notations} In the article, the sets of integers, nonnegative integers
and natural numbers are denoted by $\mathbb{Z}$, $\mathbb{Z}_+$ and $\mathbb{N}$, respectively.

\medskip

A t.d.s. $(X, T)$ is {\it transitive} if for each pair non-empty
open subsets $U$ and $V$, $N(U, V)=\{n \in \mathbb{Z}_+: T^{-n}V
\cap U \neq \emptyset\}$ is infinite; it is {\it totally
transitive} if $(X, T^n)$ is transitive for each $n \in
\mathbb{N}$; and it is {\it weakly mixing} if $(X \times X, T \times
T)$ is transitive. We say that $x \in X$ is a {\it transitive point}
if its orbit $\mathrm{orb}(x, T)=\{x, Tx, T^2x, \ldots\}$ is dense
in $X$. The set of transitive points is denoted by
$\mathrm{Tran}_T$. It is well known that if $(X, T)$ is transitive,
then $\mathrm{Tran}_T$ is a dense $G_{\delta}$-set.

\medskip

A t.d.s. $(X, T)$ is {\it minimal} if $\mathrm{Tran}_T=X$, i.e., it
contains no proper subsystems. A point $x \in X$ is called a {\it
minimal point} or {\it almost periodic point} if
$(\overline{\mathrm{orb}(x, T)}, T)$ is a minimal subsystem of $(X,
T)$. We say that $x \in X$ is a {\it periodic point} if $T^nx=x$ for
some $n \in \mathbb{N}$. The set of all periodic points
(resp. minimal points) of $(X, T)$ is denoted by $P(T)$
(resp. $AP(T)$). A t.d.s. $(X, T)$ is called
\begin{itemize}
 \item a {\it $P$-system} if it is transitive and the set of periodic points is dense;

 \item an {\it $M$-system} if it is transitive and the set of minimal points is dense;

 \item an {\it $E$-system} if it is transitive and there is an invariant Borel probability measure $\mu$ with full support, i.e., $\mathrm{supp}(\mu)=X$.
\end{itemize}

Let $S$ be a subset of $\mathbb{Z}_+$. The {\it upper Banach density} and {\it upper density} of $S$ are defined by
$$BD^*(S)=\limsup_{|I| \rightarrow \infty} \frac{|S \cap I|}{|I|} \mbox{\ \ and \ } D^*(S)=\limsup_{n \rightarrow \infty} \frac{|S \cap [0, n-1]|}{n}.$$
where $I$ is taken over all non-empty finite intervals of
$\mathbb{Z}_+$ and $|\cdot|$ denote the cardinality of the set.

\medskip

A subset $S$ of $\mathbb{Z}_+$ is {\it syndetic} if it has a bounded
gap, i.e., there is $N \in \mathbb{N}$ such that $\{i, i+1, \ldots,
i+N\} \cap S \neq \emptyset$ for every $i \in \mathbb{Z}_+$; $S$ is
{\it thick} if it contains arbitrarily long runs of positive
integers, i.e., for every $n \in \mathbb{N}$ there exists some $a_n
\in \mathbb{Z}_+$ such that $\{a_n, a_n+1, \ldots, a_n+n\} \subset
S$.

\medskip

For a t.d.s. $(X, T)$, $x \in X$ and $U \subset X$ let
$$N(x, U)=\{n \in \mathbb{Z}_+: T^nx \in U\}.$$
It is well know that $x \in X$ is a minimal point if and only if
$N(x, U)$ is syndetic for any neighborhood $U$ of $x$; a t.d.s. $(X,
T)$ is weakly mixing if and only if $N(U, V)$ is thick for any
non-empty open subsets $U, V$ of $X$ (see, for example, \cite{Fur67,
Fur81}); and a t.d.s. $(X, T)$ is an $E$-system if and only if there
is a transitive point $x \in X$ such that $N(x, U)$ has a positive
upper Banach density for any neighborhood $U$ of $x$ (see, for example, \cite[Lemma 3.6]{HPY07}).

\medskip

Let $(X, T)$ be a t.d.s. and $(x, y) \in X^2$. It is a {\it proximal pair} if
$$\liminf_{n \rightarrow +\infty} d(T^nx, T^ny)=0;$$
and it is a {\it distal pair} if it is not proximal. Denote by $P(X,
T)$ or $P_X$ the set of all proximal pairs of $(X, T)$. A point $x$
is said to be {\it distal} if whenever $y$ is in the orbit closure
of $x$ and $(x, y)$ is proximal, then $x=y$. A t.d.s. $(X, T)$ is
called {\it distal} if $(x, x')$ is distal whenever $x, x' \in X$
are distinct.

\medskip

Let $\{p_i\}_{i=1}^{\infty}$ be an infinite sequence in
$\mathbb{N}$. One  defines
$$FS(\{p_i\}_{i=1}^{\infty}):=\{p_{i_1}+p_{i_2}+\cdots+p_{i_k}: 1 \le i_1<i_2<\cdots<i_k, k \in \mathbb{N}\}.$$
A subset $F \subset \mathbb{N}$ is called an {\it IP-set} if it
contains some $FS(\{p_i\}_{i=1}^{\infty})$. A subset of $\mathbb{N}$
is called an {\it IP$^*$-set} if it has non-empty intersection with
any IP-sets. Denote by $\mathcal{F}_{ip}^*$ the family of all
IP$^*$-sets. It is known that $\mathcal{F}_{ip}^*$ is a filter,
i.e., $F_1, F_2 \in \mathcal{F}_{ip}^*$ implies $F_1 \cap F_2 \in
\mathcal{F}_{ip}^*$ (see \cite[p. 179]{Fur81}); and $x$ is distal if
and only if $x$ is IP$^*$-recurrent, i.e., $N(x, U) \in
\mathcal{F}_{ip}^*$ for any neighborhood $U$ of $x$ (see, for example,
\cite[Theorem 9.11]{Fur81} or \cite[Proposition 2.7]{DSY12}).

\subsection{Hyperspace} Let $X$ be a compact metric space with a metric $\rho$. Let $K(X)$ be the hyperspace
on $X$, that is, the space of non-empty closed subsets of $X$ equipped with the {\it Hausdorff metric} $d_H$ defined by
$$d_H(A, B)=\max \left\{\max_{x \in A} \min_{y \in B} \rho(x, y), \max_{y \in B} \min_{x \in A} \rho(x, y)\right\} \mbox{\ for } A, B \in K(X).$$
This metric turns $K(X)$ into a compact space. It is easy to see that the finite subsets of $X$ are dense in $K(X)$.

\medskip

For any non-empty open subsets $U_1, \ldots, U_n$ of $X$, $n \in \mathbb{N}$, let
$$\langle U_1, \ldots, U_n\rangle=\left\{K \in K(X): K \subset \bigcup_{i=1}^n U_i \mbox{ and } K \cap U_i \neq \emptyset \mbox{ for each } i=1, \ldots, n\right\}.$$
The following family
$$\{\langle U_1, \ldots, U_n\rangle: U_1, \ldots, U_n \mbox{ are non-empty open subsets of } X, n \in \mathbb{N}\}$$
forms a basis for the topology obtained from the Hausdorff metric $d_H$, which is called the {\it Vietoris topology} (see \cite[Theorem 4.5]{Nad92}).

\medskip

Now let $(X, T)$ be a t.d.s. The transformation $T$ induces a continuous map $T_K: K(X) \rightarrow K(X)$ defined by
$$T_K(C)=TC \mbox{ for } C \in K(X).$$
It is easy to check that $(K(X), T_K)$ is also a t.d.s.

\subsection{Probability measures spaces} Let $M(X)$ denote the space of Borel probability measures on $X$ equipped with the {\it Prohorov metric} $D$ defined by
$$D(\mu, \nu)=\inf\left\{\epsilon: \begin{array}{l l} & \mu(A) \le \nu(A^{\epsilon})+\epsilon \mbox{ and }
 \nu(A) \le \mu(A^{\epsilon})+\epsilon \mbox{ for all }\\
& \mbox{Borel subsets } A \subset X \end{array}\right\}$$ for $\mu,
\nu \in M(X)$, where $A^{\epsilon}=\{x \in X: \rho(x, A)<\epsilon\}$.
The induced topology is just the {\it weak$^*$-topology} for
measures. It turns $M(X)$ into a compact metric space. A basis is
given by the collection of all sets of the form
$$V_{\mu}(f_1, \ldots, f_k; \epsilon)=\left\{\nu \in M(X): \left|\int_X \, f_i \,
\mathrm{d}\mu-\int_X \, f_i \, \mathrm{d}\nu\right|<\epsilon, 1
\le i \le k\right\},$$ where $\mu \in M(X), k \ge
1, f_i \in C(X, \mathbb{R})$ (here $C(X, \mathbb{R})$ denote the
Banach space of continuous real-valued functions on $X$ with the
supremum norm $\|\cdot\|$) and $\epsilon>0$. If
$\{f_n\}_{n=1}^{\infty}$ is a dense subset of $C(X, \mathbb{R})$,
then
$$d(\mu, \nu)=\sum_{n=1}^{\infty} \frac{|\int \, f_n \, \mathrm{d}\mu-\int \, f_n \, \mathrm{d}\nu|}{2^n \left(\|f_n\|+1\right)}$$
is also a metric on $M(X)$ giving the weak$^*$-topology.

\begin{lem} \label{lem:2-1}
{\rm (\cite[pp. 149]{Wal82})} The following statements are equivalent:
\begin{enumerate}
  \item $\mu_n \rightarrow \mu$ in the weak$^*$-topology;
  \item For each closed subset $F$ of $X$, $\limsup \limits_{n \rightarrow \infty} \mu_n(F)\le \mu(F)$;
  \item For each open subset $U$ of $X$, $\liminf \limits_{n \rightarrow \infty} \mu_n(U)\ge \mu(U)$.
\end{enumerate}
\end{lem}

For $x \in X$, let $\delta_x \in M(X)$ denote the {\it Dirac point measure} of $x$ defined by
$$\delta_x(A)=\left\{ \begin{array}{lc} 1, & x \in A \\ 0, & x \notin A. \end{array}\right.$$
It is easy to see that the map $x \mapsto \delta_x$ imbeds $X$
inside $M(X)$. Note that $M(X)$ is convex and that the point
measures are just the extremal points of $M(X)$. It follows that the
convex combinations of point measures are dense in $M(X)$.

For $\mu \in M(X)$ and a Borel subset $A$ of $X$ with $\mu(A)>0$, the {\it conditional measure} of $A$ is defined by
$$\mu_A(B)=\frac{\mu(A \cap B)}{\mu(A)}$$
for all Borel subsets $B \subset X$.

\begin{lem} \label{lem:2-2}
Let $X, Y$ be two compact metric spaces, $\mu \in M(X)$ and $\nu \in
M(Y)$.

\begin{enumerate}
  \item  If $A=\bigcup_{i=1}^n A_i$, where $A_1, \ldots, A_n$ are Borel subsets of $X$ with $\mu(A_i)>0$ and $\mu(A_i \cap A_j)=0$
  for all $1 \le i<j \le n$, then $\mu_A=\sum_{i=1}^n \frac{\mu(A_i)}{\mu(A)} \mu_{A_i}$.
  \item Let $\epsilon>0$ and $A$ be a Borel subset of $X$ with $\mu(A)>0$. If $B$ is a Borel subset of $X$ such that
  $\mu(B)>0$ and $\mu(A \bigtriangleup B)< \mu(A) \cdot \epsilon$, then $d(\mu_A, \mu_B) \le 2\epsilon$.

  \item If $\pi: (X, \mu) \rightarrow (Y, \nu)$ is measurable and $\pi \mu=\nu$, then $\pi \mu_{\pi^{-1}A}=\nu_A$ for each Borel subset $A$ of $Y$.
\end{enumerate}
\end{lem}

\begin{proof}
(1) Note that for any Borel subsets $B$ of $X$, we have
$$\mu_A(B)=\frac{\mu(A \cap B)}{\mu(A)}=\sum_{i=1}^n \frac{\mu(A_i \cap B)}{\mu(A)}=\sum_{i=1}^n \frac{\mu(A_i)}{\mu(A)} \mu_{A_i}(B).$$

\medskip

(2) Assume that $A, B$ are two Borel subsets of $X$ such that $\mu(A) \mu(B)>0$ and $\mu(A \bigtriangleup B)< \mu(A) \cdot \epsilon$.
Then for each $f \in C(X, \mathbb{R})$ we have
\begin{eqnarray*}
\left|\int \, f \, \mathrm{d}\mu_A - \int \, f \, \mathrm{d}\mu_B\right| &=& \left|\frac{1}{\mu(A)} \int_A \, f \,
\mathrm{d}\mu-\frac{1}{\mu(B)} \int_B \, f \, \mathrm{d}\mu\right|\\
&\le& \frac{1}{\mu(A)} \cdot \left|\int_A \, f \, \mathrm{d}\mu-\int_B \, f \, \mathrm{d}\mu\right| +
 \frac{|\mu(A)-\mu(B)|}{\mu(A) \cdot \mu(B)}\cdot \left|\int_B \, f \, \mathrm{d}\mu \right| \\
&\le& \frac{2 \cdot \mu(A \bigtriangleup B)}{\mu(A)} \cdot \|f\|<2 \|f\| \cdot \epsilon.
\end{eqnarray*}
Hence,
$$d(\mu_A, \mu_B)=\sum_{n=1}^{\infty} \frac{|\int \, f_n \, \mathrm{d}\mu-\int \, f_n \,
\mathrm{d}\nu|}{2^n \left(\|f_n\|+1\right)} \le \sum_{n=1}^{\infty} \frac{\epsilon}{2^{n-1}}=2\epsilon.$$

\medskip

(3) is an obvious fact.
\end{proof}

\medskip

Let $(X, T)$ be a t.d.s. The transformation $T$ induces a continuous map $T_M: M(X) \rightarrow M(X)$ defined by
$$(T_M\mu)(A)=\mu(T^{-1}A), \mu \in M(X),  A \subset X \mbox{ Borel}.$$
It is easy to check that $(M(X), T_M)$ is also a t.d.s. For $n \in
\mathbb{N}$, define
$$M_n(X)=\left\{\frac{1}{n}\sum_{i=1}^n \delta_{x_i}: x_i \in X \mbox{ (not necessarily distinct)}\right\}.$$

\begin{lem} \label{lem:2-3}
$M_n(X)$ is closed in $M(X)$ and invariant under $T_M$. $\bigcup_{n=1}^{\infty} M_n(X)$ is dense in $M(X)$.
\end{lem}

Let $M(X, T)=\{\mu \in M(X): T_M \mu=\mu\}$ be the set of all
$T$-invariant measures and $\mathcal{E}(X, T)$ be the set of all
ergodic measures in $M(X, T)$. It is well known that $M(X, T)$ is
nonempty, compact and convex. If $M(X, T)$ consists of a single
point, then $(X, T)$ is said to be {\it uniquely ergodic}.

\subsection{Product system, factor and extenison}
For two t.d.s. $(X, T)$ and $(Y, S)$, their {\it product system} $(X \times Y, T \times S)$ is defined by
$$T \times S(x, y)=(Tx, Sy) \mbox{ for } x \in X \mbox{ and } y \in Y.$$
Higher order product systems are defined analogously and we write $(X^n, T^{(n)})$ for the $n$-fold product system $(X \times \cdots \times X, T \times \cdots \times T)$.

\medskip

Let $(X, T)$ and $(Y, S)$ be two t.d.s. A continuous map $\pi: X
\rightarrow Y$ is called a {\it homomorphism} or {\it factor map}
between $(X, T)$ and $(Y, S)$ if it is onto and $\pi \circ T=S \circ
\pi$. In this case we say $(X, T)$ is an {\it extension} of $(Y, S)$
or $(Y, S)$ is a {\it factor} of $(X, T)$. It is easy to see that
$\pi$ induces in an obvious way a homomorphism from $(M(X), T_M)$
onto $(M(Y), S_M)$ and from $(K(X), T_K)$ onto $(K(Y), S_K)$.

\section{Dynamic properties on hyperspace}

In this section, we will study some dynamic properties of $(K(X), T_K)$. Firstly, we recall the following useful lemma.

\begin{lem} \label{lem:3-1}
{\rm (see \cite{Ban05})} Let $(X, T)$ be a t.d.s.. Then the following statements are equivalent:
\begin{enumerate}
  \item $(K(X), T_K)$ is weakly mixing;
  \item $(K(X), T_K)$ is transitive;
  \item $(X, T)$ is weakly mixing.
\end{enumerate}
\end{lem}

Let $(X, T)$ be a t.d.s. We say that $(X, T)$ has {\it dense small
periodic sets} \cite{HY05} if for any non-empty open subset $U$ of
$X$ there exists a closed subset $Y$ of $U$ and $k \in \mathbb{N}$
such that $T^kY \subset Y$. Clearly, every $P$-system has dense
small periodic sets. If $(X, T)$ is transitive and has dense small
periodic sets, then it is an $M$-system.

The system $(X, T)$ is called an {\it HY-system} if it is totally
transitive and has dense small periodic sets. In \cite{HY05}, Huang
and Ye showed that an HY-system is also weakly mixing and disjoint
from any minimal systems. There exists an HY-system without periodic
points \cite[Example 3.7]{HY05}. In \cite{Li11}, the author
characterized HY-systems by transitive points via the family of
weakly thick sets. Recently, Li \cite{Li13} studied the Devaney's
chaos on the hyperspace. That is, he obtained

\begin{thm} \label{thm:3-2}
{\rm (see \cite{Li13})} $(X, T)$ is an HY-system if and only if $(K(X), T_K)$ is a $P$-system.
\end{thm}

We recall that a t.d.s. $(X, T)$ is said to be {\it pointwise
periodic} if all points in $X$ are  periodic; and it is said to be
{\it periodic} if there exists $m \in \mathbb{N}$ such that $T^m$ is
the identity map of $X$. It is clear that if $(K(X), T_K)$ is
pointwise periodic, then $(X, T)$ is also pointwise periodic.
However, the following example shows that the converse is not true.

\begin{ex} \label{ex:3-3}
Let $X=\{0\} \cup \left\{\frac{1}{n}: n \in \mathbb{N} \right\}$ with the subspace topology of the real line $\mathbb{R}$. Define $T: X \rightarrow X$ as
\begin{itemize}
 \item $T(0)=0$ and $T(1)=1$;

 \item $T\left(\frac{1}{2^n}\right)=\frac{1}{2^n+1}, \ldots, T\left(\frac{1}{2^{n+1}-1}\right)=\frac{1}{2^n}$ for each $n \in \mathbb{N}$.
\end{itemize}
Note that $(X, T)$ is pointwise periodic. Let $K=\{0\} \cup \left\{\frac{1}{2^n}: n \in \mathbb{Z}_+\right\}$. Then $K$ is not a periodic point of $(K(X), T_K)$.
\end{ex}

What we have is the following theorem and we omit the simple proof.

\begin{thm} \label{thm:3-4}
The following statements are equivalent:
\begin{enumerate}
  \item $(X, T)$ is periodic;
  \item $(K(X), T_K)$ is periodic;
  \item $(K(X), T_K)$ is pointwise periodic.
\end{enumerate}
\end{thm}

Next we study the characterization of $M$-systems and $E$-systems on the hyperspace.

\begin{thm} \label{thm:3-5}
$(X,T)$ is a weakly mixing $M$-system if and only if $(K(X),T_K)$ is an $M$-system.
\end{thm}

\begin{proof}
Let $(K(X),T_K)$ be an $M$-system. By Lemma \ref{lem:3-1}, $(X,T)$
is weakly mixing. Now we show that the set of minimal points of
$(X, T)$ is dense. Let $U,V\subset X$ be two non-empty open subsets
of $X$ with $\overline{V}\subset U$. Then $\langle V \rangle=\{A \in
K(X): A \subset V\}$ is a non-empty open subset of $K(X)$. Since
$(K(X), T_K)$ is an $M$-system, there exists a minimal point $C \in
\langle V \rangle$ of $T_K$. It follows that there exists a syndetic
subset $F$ of $\mathbb{Z}_+$ such that $T_K^n(C)\in \langle V
\rangle$, which implies that $T^n_K(C)\subset V$ for all $n\in F$.
Let $D=\overline{\bigcup_{n \in F}T^n_K(C)}$, then $D\subset
\overline{V} \subset U$. By \cite[Theorem 7]{BF02}, $D\cap AP(T)
\not=\emptyset$, so $U\cap AP(T) \not=\emptyset$. That is, $(X,T)$
is an $M$-system.

Now assume $(X, T)$ is a weakly mixing $M$-system. Then the product
system $(X^n, T^{(n)})$ is  an $M$-system for each $n \in
\mathbb{N}$. This implies that the restriction of $T_K$ to
$K_n(X)=\{C \in K(X): |K| \le n\}$, as a factor of $T^{(n)}$,
is also an $M$-system. Notice that $\bigcup_{n=1}^{\infty} K_n(X)$
is dense in $K(X)$. Hence the set of minimal points of $(K(X), T_K)$
is dense in $K(X)$. That is, $(K(X), T_K)$ is an $M$-system.
\end{proof}

\begin{thm} \label{thm:3-6}
$(X,T)$ is a weakly mixing $E$-system if and only if $(K(X),T_K)$ is an $E$-system.
\end{thm}

\begin{proof}
By \cite[Proposition 5]{BS75} and Lemma \ref{lem:3-1}, it remains to show that $(K(X),T_K)$ is an $E$-system implies that
$(X,T)$ is an $E$-system.  Assume that $(X,T)$ is not an $E$-system,
then there is a non-trivial factor $(Y,S)$ of $(X,T)$ such that
$(Y,S)$ is uniquely ergodic with a fixed point $p$.

Let $\nu'$ be an invariant probability measure on $K(Y)$ with full
support. Assume $U \subset Y$  is non-empty open and $p\not \in
\overline{U}$. Since $\langle U \rangle$ is a non-empty open subset
of $K(Y)$, there is an ergodic measure (using ergodic decomposition)
$\nu$ on $K(Y)$ with $\nu (\langle U\rangle)>0$. Let $C \in \langle
U \rangle$ be a generic point for $\nu$ (i.e.,
$\frac{1}{n}\sum_{i=0}^{n-1} \delta_{T_K^iC} \rightarrow \nu$), then
the return time set $N(C, \langle U \rangle)=\{n\in\Z:T^nC\subset
U\}$ has positive upper density. This implies that each point $x \in
C$ returns to $U$ with positive upper density. Fix $x \in C$. Since
$(Y,S)$ is uniquely ergodic, $\frac{1}{n}\sum_{i=0}^{n-1}
\delta_{T^ix}\lra \delta_p$ and thus by Lemma \ref{lem:2-1} we have
$$0=\delta_p(\overline{U}) \ge \limsup_{n \rightarrow \infty}
\frac{1}{n}\sum_{i=0}^{n-1} \delta_{T^ix} (\overline{U})=D^*(N(x,
\overline{U}))>0,$$  a contradiction. This shows that $(X,T)$ is an
$E$-system.
\end{proof}

\section{Dynamic properties on the space of probability measures}

In this section, we will study some dynamic properties of $(M(X), T_M)$.

\subsection{Distal points and minimal points} In \cite{BS75}, the authors showed the distality (resp. minimality)
of $(X, T)$ does not necessarily imply the distality (resp. minimality) of $(M(X), T_M)$. But we have the following result.

\begin{thm} \label{thm:4-1}
If $(X, T)$ is distal, then the set of distal points of $(M(X),
T_M)$ is dense in $M(X)$. If $(X, T)$ is an $M$-system,  then the
set of minimal points of $(M(X), T_M)$ is dense.
\end{thm}

\begin{proof}
Assume that $(X, T)$ is a distal system. Then the product system
$(X^n, T^{(n)})$ is distal for every $n \in \mathbb{N}$.  This
implies that the restriction of $T_M$ to $M_n(X)$, as a factor of
$T^{(n)}$, is also distal. Notice that each $\mu \in M_n(X)$ is a
distal point of $(M(X), T_M)$. Hence the set of distal points of
$(M(X), T_M)$ is dense in $M(X)$.

Now assume $(X, T)$ is an $M$-system. Then the set of minimal points
of the product system $(X^n, T^{(n)})$  is dense for each $n \in
\mathbb{N}$. This implies that the restriction of $T_M$ to $M_n(X)$,
as a factor of $T^{(n)}$, is also has dense minimal points. Notice
that $\bigcup_{n=1}^{\infty} M_n(X)$ is dense in $M(X)$. Hence the
set of minimal points of $(M(X), T_M)$ is dense in $M(X)$.
\end{proof}

\subsection{Weakly mixing} In this subsection, we study the weakly mixing property of
$(M(X), T_M)$. Firstly, we recall the following useful lemma.

\begin{lem} \label{lem:4-2}
\cite{Pet70} $(X, T)$ is weakly mixing if and only if $N(U, U) \cap N(U, V) \neq \emptyset$ for any non-empty open subsets $U, V \subset X$.
\end{lem}

The following fact is known, see for example \cite{Shao}. We give a
proof for completeness.

\begin{thm} \label{thm:4-3}
Let $(X, T)$ be a t.d.s.. Then the following statements are equivalent:
\begin{enumerate}
  \item $(X, T)$ is weakly mixing;
  \item $(M(X), T_M)$ is weakly mixing;
  \item $(M(X), T_M)$ is transitive.
\end{enumerate}
\end{thm}

\begin{proof}
(1) $\Rightarrow$ (2) can see \cite[Theorem 1]{BS75}; and (2) $\Rightarrow$ (3) is obvious.

\medskip

(3) $\Rightarrow$ (1) Let $U, V$ be two non-empty open subsets of $X$ and let $W_1=\{\mu \in M(X):
\mu(U)>\frac{4}{5}\}$ and $W_2=\{\mu \in M(X): \mu(U)>\frac{2}{5} \mbox{ and } \mu(V)>\frac{2}{5}\}$.
It is clear that $W_i$ is non-empty open in $M(X)$ for $i=1, 2$.

By the transitivity of $T_M$, there is $k \in \mathbb{N}$ such that
$W_1 \cap T_M^{-k} W_2 \neq \emptyset$.  Let $\mu \in W_1 \cap
T_M^{-k} W_2$. Then we have $\mu \in W_1$ and $T^k \mu \in W_2$,
which implies that $U \cap T^{-k}U \neq \emptyset$ and $U \cap
T^{-k}V \neq \emptyset$. By Lemma \ref{lem:4-2}, $(X, T)$ is weakly
mixing.
\end{proof}

\subsection{$E$-system} In this subsection, we will study the characterization of $E$-systems on the space
of probability measures. That is, we have the following result.

\begin{thm} \label{thm:4-4}
$(X,T)$ is a weakly mixing $E$-system if and only if $(M(X),T_M)$ is an $E$-system.
\end{thm}

\begin{proof}
Assume that $(X,T)$ is a weakly mixing $E$-system. It remains to
show that there is a $T_M$-invariant measure $\mu \in M(M(X))$ with
full support. It is easy to see that $T^{(n)}$ admits an invariant
measure with full support. This implies that the restriction of
$T_M$ to $M_n(X)$, as a factor of $T^{(n)}$, admits an invariant
measure $\mu_n$ with full support on $M_n(X)$. Since
$\bigcup_{n=1}^{\infty} M_n(X)$ is dense in $M(X)$, the
$T_M$-invariant measure $\mu=\sum_{n=1}^{\infty} \frac{1}{2^n} \mu_n
\in M(M(X))$ with $\mathrm{supp}(\mu)=M(X)$. That is, $(M(X), T_M)$
is an $E$-system.

Now assume that $(M(X), T_M)$ is an $E$-system. By Theorem
\ref{thm:4-3}, $(X, T)$ is weakly mixing.  Let $\nu$ be a
$T_M$-invariant measure on $M(X)$ with full support. Then the
barycenter $\mu=\int_{M(X)} \, \theta \, \mathrm{d}\nu(\theta)$ is a
$T$-invariant measure on $X$. Let $U$ be a non-empty subset of $X$
and let $V=\{m \in M(X): m(U)
>\frac{1}{2}\}$. Then we have
$$\mu(U) \ge \int_V \, \theta(U) \, \mathrm{d}\nu(\theta) \ge \frac{1}{2} \nu(V)>0.$$
This implies $\mathrm{supp}(\mu)=X$.
\end{proof}

\subsection{$P$-system} Note that for each $x \in X$, $\delta_x$ is a periodic point of $(M(X), T_M)$ which
implies that $x$ is a periodic point of $(X, T)$. Hence, if $(M(X), T_M)$ is pointwise periodic,
then $(X, T)$ is also pointwise periodic. However, the following example shows that the converse is not true.

\begin{ex} \label{ex:4-5}
Let $X=\{0\} \cup \left\{\frac{1}{n}: n \in \mathbb{N} \right\}$ with the subspace topology of the real line $\mathbb{R}$. Define $T: X \rightarrow X$ as
\begin{itemize}
 \item $T(0)=0$ and $T(1)=1$;

 \item $T\left(\frac{1}{2^n}\right)=\frac{1}{2^n+1}, \ldots, T\left(\frac{1}{2^{n+1}-1}\right)=\frac{1}{2^n}$ for each $n \in \mathbb{N}$.
\end{itemize}
Note that $(X, T)$ is pointwise periodic. Let $\mu=\sum_{n=1}^{\infty}\frac{1}{2^n}\delta_{\frac{1}{2^{n-1}}}$. Then $\mu \in M(X)$ is not periodic.
\end{ex}

Again we omit the simple proof of the following fact.

\begin{thm} \label{thm:4-6}
The following statements are equivalent:
\begin{enumerate}
  \item $(X, T)$ is periodic;
  \item $(M(X), T_M)$ is periodic;
  \item $(M(X), T_M)$ is pointwise periodic.
\end{enumerate}
\end{thm}

In order to characterize $P$-systems on the space of probability
measures, we need a notion of an almost HY-system.

\begin{de} \label{de:4-7}
Let $(X,T)$ be a t.d.s. We say that $(X,T)$ has {\it almost dense
periodic sets} if for each non-empty  open subset $U \subset X$ and
$\epsilon>0$, there are $k \in \N$ and $\mu \in M(X)$ with
$T_M^k\mu=\mu$ such that $\mu(U^c)<\epsilon$, where $U^c=\{x \in X:
x \notin U\}$. We say that $(X, T)$ is an {\it almost HY-system} if
it is totally transitive and has almost dense periodic sets.
\end{de}

We note that $(X,T)$ has almost dense periodic sets if and only if
for each non-empty open subset $U \subset X$, there are periodic points
$\mu_k\in M(X)$ such that $\mu_k(U)\rightarrow 1$. It is easy to see that
if $(X, T)$ has dense small periodic sets, then it has almost dense
periodic sets, and hence an HY-system is a weakly mixing almost HY-system.
If $(X, T)$ has almost dense periodic sets, then it has
an invariant measure with full support. Moreover, we also have the
following result.

\begin{prop} \label{rem:4-8}
Let $(X, T)$ and $(Y, S)$ be two t.d.s.
\begin{enumerate}
\item If $(X, T)$ has almost dense periodic sets and $\pi: (X, T) \rightarrow (Y, S)$ is a factor map, then $(Y, S)$ has almost dense periodic sets.

\item If $(X, T)$ and $(Y, S)$ have almost dense periodic sets, then $(X \times Y, T \times S)$ has almost dense periodic sets.
\end{enumerate}
\end{prop}

\begin{proof}
(1) Let $V \subset Y$ be a non-empty open subset and $\epsilon>0$.
Then $U=\pi^{-1}(V) \subset X$ is open and non-empty, and there is
$T^k$-invariant measure $\mu$ with $\mu(U^c)<\epsilon$. This implies
that $\pi\mu(V^c)<\ep$. It is clear that $\nu=\pi\mu$ is
$S^k$-invariant, and hence $(Y, S)$ has almost dense periodic sets.

\medskip

(2) Let $U$ be a non-empty open subset of $X \times Y$  and
$\epsilon>0$. Then there are non-empty open subsets $U_1 \subset X$
and $U_2 \subset Y$ such that $U_1\times U_2\subset U$. Since $(X,
T)$ and $(Y, S)$ have almost dense periodic sets, there are
$T^{k_1}$-invariant measure $\mu \in M(X)$ and $S^{k_2}$-invariant
measure $\nu \in M(Y)$ with $\mu(U_1^c)<\tfrac{1}{2}\ep$ and
$\nu(U_2^c)<\tfrac{1}{2}\ep$. Set $k=k_1 \times k_2$, then $\mu
\times \nu \in M(X \times Y)$ is $T^k \times S^k$-invariant and
$$(\mu \times \nu)(U^c)\le (\mu \times \nu)((U_1\times U_2)^c) \le \mu(U_1^c)+\nu(U_2^c)<\epsilon.$$
That is, $(X \times Y, T \times S)$ has almost dense periodic sets.
\end{proof}

It is well known that a totally transitive $P$-system is weakly
mixing \cite{Ban97}. In \cite{HY05}, Huang  and Ye showed that a
totally transitive system with dense small periodic sets is weakly
mixing. Now we improve these results by showing that each almost
HY-system is weakly mixing. That is,

\begin{prop} \label{prop:4-9}
If $(X, T)$ is a totally transitive system with almost dense
periodic sets, then $(X, T)$ is weakly mixing.
\end{prop}

\begin{proof}
Let $U, V$ be two non-empty open subsets of $X$. Since $(X, T)$ has
almost dense periodic sets, there exist $k \in \mathbb{N}$ and a
$T^k$-invariant measure $\mu \in M(X)$ such that
$\mu(U)>\frac{4}{5}$. Hence, $\{ki: i \in \mathbb{Z}_+\} \subset
N(U, U)$. Note that $(X, T^k)$ is topological transitive, and thus
$N(U, U) \cap N(U, V) \neq \emptyset$. By Lemma \ref{lem:4-2}, $(X,
T)$ is weakly mixing.
\end{proof}

\begin{thm} \label{thm:4-10}
The following statements are equivalent:
\begin{enumerate}
\item $(X,T)$ is an almost HY-system.

\item $(X,T)$ is weakly mixing, and for each open non-empty $U\subset X$ and $\epsilon>0$
there are $k\in\N$ and $\mu\in \mathcal{E}(X, T^k)$ such that $\mu(U^c)<\ep$.

\item $(M(X),T_M)$ is a $P$-system.

\item $(M(X),T_M)$ is an almost HY-system.
\end{enumerate}
\end{thm}

\begin{proof}
(1) $\Rightarrow$ (2) Let $(X, T)$ be an almost HY-system. By Proposition \ref{prop:4-9}, $(X, T)$ is weakly mixing.
For each non-empty open subset $U \subset X$ and $\epsilon>0$,
there are $k \in \N$ and $\nu \in M(X)$ with $T_M^k\nu=\nu$ such
that $\nu(U^c)<\epsilon$. Using the ergodic decomposition we have
$\nu=\int_{\mathcal{E}(X, T^k)} \, m \, \mathrm{d}\tau (m)$, where
$\tau \in M(M(X))$ with $\tau(\mathcal{E}(X, T^k))=1$. It follows
that
$$\int_{\mathcal{E}(X, T^k)} \, m(U^c) \, \mathrm{d}\tau (m)=\nu(U^c)<\ep.$$
Thus there is $\mu\in \mathcal{E}(X, T^k)$ such that $\mu(U^c)<\epsilon$.

\medskip

(2) $\Rightarrow$ (3) Let $x\in X$. For any $f_1, \ldots, f_s\in
C(X,\R)$ and $\epsilon>0$,  there is $r>0$ such that if
$d(y,x)<r<\epsilon/2$ then $|f_j(y)-f_j(x)|<\epsilon/2$ for all
$j=1, \ldots, s$. Let $U=B(x,r)$. Then there is $k \in \mathbb{N}$
and $\mu \in \mathcal{E}(X, T^k)$ such that
$\mu(U^c)<(\max\{||f_j||:{1\le j\le
s}\})^{-1}\epsilon/4$. Thus
$$\left|\int_X f_j(y)d\mu(y)-f_j(x)\right|\le \int_{U^c}|f_j(y)-f_j(x)|d\mu(y)+\int_{U}|f_j(y)-f_j(x)|d\mu(y) <\epsilon,$$
for each $1\le j \le s$. That is, $\mu \in V_{\delta_x}(f_1,\ldots,f_s;\epsilon)$. So $\delta_x$ is a limit point of $P(T_M)$.

For $x_1, \ldots, x_n\in X$, let $\nu=\tfrac{1}{n}\sum_{j=1}^n
\delta_{x_j}$. For give $f_1,\ldots,f_s\in C(X,\R)$ and
$\epsilon>0$, let $\mu_j\in P(T_M)\cap
V_{\delta_{x_j}}(f_1,\ldots,f_s;\ep)$ and
$\mu=\tfrac{1}{n}\sum_{j=1}^n \mu_j$. It is clear that $\mu$ is a
periodic point of $T_M$ and $\mu\in V_{\nu}(f_1,\ldots,f_s;\ep)$.
This implies that $M_n(X)\subset \overline{P(T_M)}$. Therefore, the
set of periodic points of $(M(X), T_M)$ is dense in $M(X)$.

\medskip

(3) $\Rightarrow$ (4) is obvious.

\medskip

(4) $\Rightarrow$ (1) Let $U$ be a non-empty open subset of $X$ and
$\epsilon>0$. Let $V=\{m\in M(X):m(U)>1-\tfrac{1}{2}\epsilon\}$.
Then there are $k\in \N$ and $T_M^k$-invariant measure $\nu$ on
$M(X)$ with $\nu(V^c)<\tfrac{1}{2}\epsilon$. Let $\mu=\int_{M(X)}\,
\theta \, \mathrm{d}\nu(\theta)$. It is clear that $T_M^k\mu=\mu$
and
$$\mu(U^c)=\int_V \, \theta(U^c) \, \mathrm{d}\nu(\theta)+ \int_{V^c} \, \theta(U^c) \, \mathrm{d} \nu(\theta)\le
\frac{\epsilon}{2}\nu(V)+\nu(V^c)<\epsilon.$$ That is, $(X, T)$ is
an almost HY-system.
\end{proof}

Let $\Sigma_2=\prod_{i=1}^{\infty} \{0, 1\}$, where $\{0, 1\}$ and
$\Sigma_2$ are equipped with the  discrete and the product topology
respectively. For $n \in \mathbb{N}$ and $(x_1, \cdots, x_n) \in
\{0, 1\}^n$, define
$$[x_1, \cdots, x_n]:=\{y \in \Sigma_2: y_i=x_i, i=1, \ldots, n\},$$
which is called an {\it $n$-cylinder} . It is known that the set of
all cylinders form a  semi-algebra which generates the Bore
$\sigma$-algebra of $\Sigma_2$. If $x=(x_1, x_2, \cdots)$ and
$y=(y_1, y_2, \cdots)$ are two elements of $\Sigma_2$, then their
sum $x \oplus y=(z_1, z_2, \cdots)$ is defined as follows. If
$x_1+y_1<2$, then $z_1=x_1+y_1$; if $x_1+y_1 \ge 2$, then
$z_1=x_1+y_1-2$ and we carry $1$ to the next position. The other
terms $z_2, \cdots$ are successively determined in the same fashion.
Let $T: \Sigma_2 \rightarrow \Sigma_2$ be defined by $T(z)=z \oplus
\mathbf{1}$ for each $z \in \Sigma_2$, where $\mathbf{1}=(1, 0, 0,
\cdots)$. It is known that $T$ is minimal and unique ergodic, which
is called a {\it dyadic adding machine}. Now we have

\begin{thm} \label{thm:4-11}
There is a minimal weakly mixing almost-HY-system.
\end{thm}

\begin{proof}
Let $(\Sigma_2, T)$ be the dyadic adding machine with a unique
ergodic measure $\mu$.  A remarkable result due to Lehrer
\cite{Leh87} which generalized the famous Jewett-Krieger Theorem
says that $(\Sigma_2, T, \mu)$ has a minimal weakly mixing uniquely
ergodic model, i.e., there exists a system $(Y, S, \nu)$ isomorphic
to $(\Sigma_2, T, \mu)$, where $(Y, S)$ is a minimal, unique ergodic
and weak mixing t.d.s. We shall show that $(Y, S)$ is an
almost-HY-system. By Theorem \ref{thm:4-10}, it remains to show that
the set of periodic points of $(M(Y), S_M)$ is dense.

Let $\pi: (\Sigma_2, T, \mu) \rightarrow (Y, S, \nu)$ be an
isomorphism, that is, there are invariant Borel subsets $X_1
\subset \Sigma_2$ and $X_2 \subset Y$ with $\mu(X_1)=\nu(X_2)=1$ and an
invertible measure-preserving transformation $\pi: X_1 \rightarrow
X_2$ such that $\pi(Tx)=S\pi(x)$ for all $x \in X_1$.

Let $\epsilon>0$ and let $U$ be a non-empty open subset of $Y$.
Since $(Y, S)$ is minimal and unique ergodic, we have $\nu(U)>0$.
Thus, there are finitely many pairwise disjoint cylinders $A_1,
\ldots, A_k$ of $\Sigma_2$ such that $\mu(\pi^{-1}U \bigtriangleup A)<
\nu(U) \cdot \epsilon$ with $A=\bigcup_{i=1}^k A_i$, which implies
$\nu(U \bigtriangleup \pi(A \cap X_1))<\nu(U) \cdot \epsilon$. Using
Lemma \ref{lem:2-2} (2), $d(\nu_U, \nu_{\pi(A \cap X_1)})\le
2\epsilon$. Since $T^{2^{|C|}}C=C$ for each cylinder $C$ of $X$,
where $|C|$ stands for the length of $C$, we conclude that $\mu_C$ is
periodic. In particular, each $\mu_{A_i}$ is periodic. By Lemma
\ref{lem:2-2} (3), each $\nu_{\pi(A_i \cap X_1)}$ is also periodic.
By Lemma \ref{lem:2-2} (1), $\nu_{\pi(A \cap X_1)}=\sum_{i=1}^k p_i
\nu_{\pi(A_i \cap X_1)}$, where $p_i=\mu(A_i)/\mu(A)$. Thus,
$\nu_{\pi(A \cap X_1)}$ is periodic. It follows that $\nu_U$ is
approached by periodic points of $(M(Y), S_M)$.

Now take $y \in Y$ and let $\{U_n\}_{n=1}^{\infty}$ be a sequence of
open neighborhoods  of $y$ such that $\mathrm{diam}(U_n) \rightarrow
0$. For any $f \in C(Y, \mathbb{R})$, we have
$$\left|\int_Y \, f(z) \, \mathrm{d}\nu_{U_n}-f(y)\right|\le \int_{U_n} \left|f(z)-f(y)\right| \mathrm{d}\nu_{U_n} \rightarrow 0.$$
A simple calculation  shows $\nu_{_{U_n}} \rightarrow \delta_y$, and
hence $\delta_y$ is a limit point  of $P(S_M)$. This implies that
each element of $M_n(Y)$ is approached by elements of $P(S_M)$.
Since $\bigcup_{n=1}^{\infty} M_n(Y)$ is dense in $M(Y)$, it follows
that $(M(Y), S_M)$ is a $P$-system.
\end{proof}

\begin{rem} \label{rem:4-12}
(1) For the dyadic adding machine $(\Sigma_2, T)$, we remark that
the convex combinations of conditional measures $\mu_B$ ($B$ is a
cylinder of $\Sigma_2$) is dense in $M(\Sigma_2)$. More precisely,
for each $n \in \mathbb{N}$, define
$C_n(\Sigma_2)=\{\mu=\frac{1}{n}\sum_{i=1}^{n} \mu_{A_i}: A_i \mbox{
is a cylinder of } \Sigma_2\}$. Then each $\mu \in C_n(\Sigma_2)$ is
periodic and $\bigcup_{n=1}^{\infty} C_n(\Sigma_2)$ is dense in
$M(\Sigma_2)$. This shows that the set of periodic points of
$(M(\Sigma_2), T_M)$ is dense.

\medskip

(2) In \cite[Theorem 5.10]{HLY12}, Huang, Li and Ye showed that a
minimal t.d.s. has dense small periodic sets if and only if it is an
almost one-to-one extension of some adding machine. We remark that a
nontrivial adding machine is never totally transitive, and hence the
systems described in Theorem \ref{thm:4-11} do not have dense small
periodic sets. This shows that any minimal almost HY-system is not an HY-system.

\end{rem}

The following question is natural.

\begin{ques} Is there a weakly mixing, proximal
$E$-system $(X,T)$ such that $(M(X),T_M)$ is a $P$-system?
\end{ques}

This question will be answered affirmatively
in a forthcoming paper of Lian, Shao and Ye by showing that each ergodic system  has a weakly mixing
and proximal topological model \cite{LSY13}. And it follows that $(X,T)$ needs not to be
a weakly mixing $M$-system whenever $(M(X),T_M)$ is an $M$-system.

\section{Disjointness}

In this section, we discuss disjointness. The notion of disjointness
of two t.d.s.  was introduced by Furstenberg in his seminar paper
\cite{Fur67}. Let $(X, T)$ and $(Y, S)$ be two t.d.s. We say $J
\subset X \times Y$ is a {\it joining} of $X$ and $Y$ if $J$ is a
non-empty closed invariant set, and is projected onto $X$ and $Y$,
respectively. If each joining is equal to $X \times Y$, then we say
that $(X, T)$ and $(Y, S)$ are {\it disjoint}, denoted by $(X, T)
\perp (Y, S)$ or $X \perp Y$. Note that if $(X, T) \perp (Y, S)$,
then one of them is minimal \cite{Fur67}, and if $(Y, S)$ is
minimal, then $(X, T)$ has dense minimal points and it is weakly
mixing if it is transitive \cite{HY05}.

\begin{rem} \label{rem:5-1}
For any nontrivial t.d.s. $(X, T)$ we have at least the two distinct
selfjoinings  $X \times X$ and $\bigtriangleup=\{(x, x): x \in X\}$,
and thus the only system which is disjoint from itself is the
trivial t.d.s. Using Theorem \ref{thm:4-11}, we obtain that there is a minimal almost HY-system is not in
$\mathcal{M}^{\perp}$, where $\mathcal{M}^{\perp}$ denote the
collection of all t.d.s. disjoint from all minimal systems.
\end{rem}

We say that $(X,T)$ is {\it strongly disjoint from all minimal systems} if $(X^n, T^{(n)})$
is disjoint from all minimal systems for any $n\in\N$. Then we have

\begin{thm} \label{thm:5-2}
Let $(X, T)$ be a t.d.s. Then
\begin{enumerate}
\item If $(K(X),T_K)$ is weakly mixing and is disjoint from all minimal systems, then $(X,T)$ is weakly mixing and is disjoint from all minimal systems.

\item  If $(X,T)$ is strongly disjoint from all minimal systems, then both $(K(X),T_K)$ and $(M(X),T_M)$ are disjoint from all minimal systems.
\end{enumerate}
\end{thm}

\begin{proof}
(1) Let $(Y, S)$ be a minimal system. Since $(K(X), T_K) \perp (Y,
S)$, we have  $\ol{\mathrm{orb}((E,y),T_K\times S)}=K(X)\times Y$
for any transitive points $E$ of $(K(X),T_K)$ and any $ y\in Y$.
Choose $x\in E$ and let $J\subset X\times Y$ be a joining. Then
there is $y\in Y$ with $(x,y)\in J$. We will show that
$\ol{\mathrm{orb}((x,y), T\times S)}=X\times Y$, which implies $J=X
\times Y$, and hence $(X,T) \perp (Y, S)$.

In fact, for any pair $u\in X$ and $v\in Y$, there is a positive integers sequence $\{n_i\}_{i=1}^{\infty}$ such that
$(T_K \times S)^{n_i}(E, y) \rightarrow (\{u\}, v)$. By the definition of the Hausdorff metric on $K(X)$ we known that
$$d(T^{n_i}x, u) \le
 d_H(T_K^{n_i}E, \{u\}) \rightarrow 0.$$
This implies $T^{n_i}x \rightarrow u$, and hence $(T \times S)^{n_i} (x, y) \rightarrow (u, v)$.

\medskip

(2) Assume that $(X, T)$ is strongly disjoint from all minimal
systems. Then the restriction of  $T_K$ to $K_n(X)$, as a factor of
$T^{(n)}$, is disjoint from all minimal systems for each $n \in
\mathbb{N}$. Since $\bigcup_{n=1}^{\infty} K_n(X)$ is dense in $K(X)$, $(K(X),
T_K)$ is disjoint from all minimal systems.

The same argument shows that $(M(X), T_M)$ is disjoint from all minimal systems.
\end{proof}

\begin{rem} \label{rem:5-3}
(1) Furstenberg \cite{Fur67} showed that each weakly mixing system
with dense periodic points is disjoint from any minimal systems. It
directly follows from Theorem \ref{thm:5-2} and Theorem
\ref{thm:3-2} that each HY-system (i.e., weakly mixing system with
dense small periodic sets) is disjoint from any minimal systems
(first proved by Huang and Ye in \cite[Theorem 3.4]{HY05}).

\medskip

(2) It follows from Theorem \ref{thm:4-11} that $(X, T)$ needs not
to be disjoint from all minimal systems whenever $(M(X), T_M)$ is
disjoint from all minimal systems.
\end{rem}

We say that a t.d.s. $(X, T)$ has {\it dense distal sets} if for
each non-empty open  subset $U$ of $X$, there is a distal point $C$
of $(K(X), T_K)$ such that $C \subset U$.

\begin{prop} \label{prop:5-4}
The following statements are equivalent:
\begin{enumerate}
\item $(X, T)$ is a weakly mixing system with dense distal sets;

\item $(K(X), T_K)$ is a weakly mixing system with dense distal points;

\item $(K(X), T_K)$ is a weakly mixing system with dense distal sets.
\end{enumerate}
\end{prop}

\begin{proof}
(1) $\Rightarrow$ (2) By Lemma \ref{lem:3-1}, $(K(X), T_K)$ is weakly
mixing. Let $n \in \mathbb{N}$ and $U_1, \ldots, U_n$
be non-empty  open subsets of $X$. Since $(X, T)$ has dense distal
sets, there exist distal points $C_1, \ldots, C_n$ of $(K(X), T_K)$
such that $C_i \subset U_i$. Let $C=\bigcup_{i=1}^n C_i$. Clearly, $C \in \langle U_1, \ldots,
U_n\rangle$. We will show that $C$ is a distal point of $(K(X), T_K)$,
which implies $(K(X), T_K)$ has dense distal points.

Let $V_1, \ldots, V_m$ be non-empty open sets of $X$ with $C \in
\langle V_1, \ldots, V_m \rangle$.  Then for each $C_i$ there are
$s_1, \ldots, s_{m_i} \in \{1, \ldots, m\}$ such that $C_i \in
\langle V_{s_1}, \ldots, V_{s_{m_i}} \rangle$. Clearly, we have
$\bigcup_{i=1}^{n} \{V_{s_1}, \ldots, V_{s_{m_i}}\}=\{V_1, \ldots,
V_m\}$. Notice that each $C_i$ is distal and
$\mathcal{F}_{\mathrm{ip}}^*$ is a filter, it is not hard to verify
that $N(C, \langle V_1, \ldots, V_m \rangle)$ contains an IP$^*$-set
$\bigcap_{i=1}^n N(C_i, \langle V_{s_1}, \ldots, V_{s_{m_i}}
\rangle)$, which implies that $C$ is a distal point of $(K(X),
T_K)$.

\medskip

(2) $\Rightarrow$ (3) is obvious.

\medskip

(3) $\Rightarrow$ (1) By Lemma \ref{lem:3-1}, $(X, T)$ is weakly
mixing. Now we show that $(X, T)$ has dense distal sets. Let $U$ be
a non-empty open subset of $X$. Since $(K(X), T_K)$ has dense distal
sets, there exists a distal point $\mathcal{A}$ of $(K(K(X)), T_K)$
such that $\mathcal{A} \subset \langle U \rangle$.
Let $C=\bigcup\{A: A \in \mathcal{A}\}$. Clearly, $C \subset U$.
Next, we will show that $C \in K(X)$. In fact, let $\{x_n\}_{n=1}^{\infty}$
be a sequence of $C$ and $x_n \rightarrow x$. Then for each $n \in \mathbb{N}$
there is $A_n \in \mathcal{A}$ such that $x_n \in A_n$. Since $\mathcal{A}$ is a
non-empty closed subset of $K(X)$ and $K(X)$ is a compact metric
space, without loss of generality, we may assume $A_n$ convergence
to $A \in \mathcal{A}$. By the definition of Hausdorff metric, we
have $x \in A \subset C$, which implies $C\in K(X)$. To complete the proof,
it remains to show that $C$ is a distal point of $(K(X), T_K)$.

Let $V_1, \ldots, V_m$ be non-empty open sets of $X$ with $C \in
\langle V_1, \ldots, V_m \rangle$.  Then there exist non-empty open
subsets $\mathcal{V}_1, \ldots, \mathcal{V}_{s}$ of $K(X)$ satisfies
that
\begin{itemize}
\item $\mathcal{V}_i=\left\langle V_{k_1^i}, \ldots, V_{k_{s(i)}^{i}}\right\rangle$ for each $i=1, \ldots, s$;

\item $\mathcal{A} \in \langle \mathcal{V}_1, \ldots, \mathcal{V}_s\rangle$;

\item $\bigcup_{i=1}^s \left\{V_{k_1^i}, \ldots, V_{k_{s(i)}^i}\right\}=\{V_1, \ldots, V_m\}$.
\end{itemize}
It follows that $N(\mathcal{A}, \langle \mathcal{V}_1, \ldots,
\mathcal{V}_s \rangle) \subset N(C, \langle V_1, \ldots, V_m
\rangle)$, and hence $C$ is IP$^*$-recurrent (see Section 2). That
is, $C$ is a distal point of $(K(X), T_K)$.
\end{proof}

In \cite{HY05} it was showed that each weakly mixing system with
dense regular  minimal points is disjoint from any minimal systems;
and in \cite{DSY12, O10}, the authors showed that each weakly mixing
system with dense distal points is disjoint from any minimal
systems. Now we improve these results by showing that each weakly
mixing system with a dense set of distal sets is disjoint from all
minimal systems. That is, we have

\begin{thm} \label{thm:5-5}
If $(X, T)$ is a weakly mixing t.d.s. and has dense distal sets, then $(X, T)$ is disjoint from all minimal systems.
\end{thm}

\begin{proof}
It directly follows from Theorem \ref{thm:5-2}, Proposition \ref{prop:5-4} and \cite[Theorem 7.14]{DSY12}.
\end{proof}

We strongly believe that there is a t.d.s. which has dense distal
sets and does not have dense distal points.


\end{document}